\documentclass[12pt, reqno]{amsart}    
\usepackage{amsmath,amssymb,amsthm,amstext}
\usepackage{enumerate}
\usepackage{graphicx}
\usepackage{amsfonts}
\usepackage[hidelinks]{hyperref}
\usepackage{multicol}
\usepackage{epsfig}
\usepackage{float}
\usepackage[section]{placeins}
\usepackage{epstopdf}
\usepackage[utf8]{inputenc}
\usepackage[T1]{fontenc}
\usepackage{dutchcal}
\usepackage{mathtools}
\usepackage{xcolor}
\usepackage[margin=1 in]{geometry}
\usepackage{subcaption}
\usepackage{comment}
\usepackage{float}
\usepackage{titlesec}
\usepackage{lineno}
\usepackage[normalem]{ulem}

\titleformat{\section}
  {\normalfont\normalsize\bfseries\centering}
  {\thesection.}
  {1em}
  {}

\titleformat{\subsection}
  {\normalfont\large\bfseries}
  {\thesubsection}
  {1em}
  {}

\titleformat{\subsubsection}
  {\normalfont\normalsize\bfseries}
  {\thesubsubsection}
  {1em}
  {}

\usepackage{chngcntr}
\numberwithin{equation}{section}
\newtheorem{thm}{Theorem}[section]
\newtheorem{rem}{Remark}[section]
\newtheorem{cor}[thm]{Corollary}

\newtheorem{lem}{Lemma}[section]

\newtheorem{prop}[thm]{Proposition}
\theoremstyle{remark}
\newtheorem{defi}{\textbf{Definition} }[section]

\newcommand{\de}{\delta}

\newcommand{\Om}{\Omega}

\newcommand{\p}{\partial}

\title[]{
{Existence results for quasimonotone elliptic systems with growth up to critical exponents}}

\makeatletter
\def\@settitle{%
  \begin{center}%
    \normalfont\normalsize\bfseries
    \let\MakeUppercase\relax
    \let\uppercase\relax
    \let\uppercasenonmath\relax
    \@title
  \end{center}%
}
\makeatother

{\small
\author[S.~Bandyopadhyay]{Shalmali Bandyopadhyay}
\address{S. Bandyopadhyay \newline 
Department of Mathematics and Statistics, The University of Tennessee at Martin, Martin, TN, USA}
\email{sbandyo5@utm.edu\\https://orcid.org/0000-0003-3971-6033}

\author[B.~B.~Delgado]{Briceyda B. Delgado}
\address{B.~B.~Delgado \newline  
INFOTEC, Centro de Investigación e Innovación en Tecnologías \\de la Información y Comunicación, Aguascalientes, Mexico}
\email{briceyda.delgado@infotec.edu.mx \\ https://orcid.org/0000-0002-1280-4650}

\author[N.~Mavinga]{Nsoki Mavinga}
 \address{N.~Mavinga \newline 
 Department of Mathematics and Statistics, Swarthmore College, Swarthmore, PA, USA}
 \email{nmaving1@swarthmore.edu\\ https://orcid.org/0000-0002-0305-6670}

 \author[M. A.~Onyido]{Maria Amarakristi Onyido}
 \address{M. A.~Onyido \newline 
 Department of Mathematics, Northern Illinois University, Dekalb, IL, USA}
 \email{monyido@niu.edu\\
 https://orcid.org/0000-0001-9675-2095}
 \thanks{Corresponding author: Briceyda B.~Delgado (briceyda.delgado@infotec.edu.mx)}
}

\begin{document}

%\linenumbers

%%For the system $-\Delta u_i + c_i(x)u_i = f_i(x, u_1, u_2)$ in $\Omega$ with $\frac{\partial u_i}{\partial \eta} = g_i(x, u_1, u_2)$ on $\partial\Omega$ for $i = 1, 2$, 

%We prove the existence results under two distinct frameworks. 
\begin{abstract}
We establish the existence of weak solutions of coupled systems of elliptic partial differential equations with quasimonotone nonlinearities in the differential equation and on the boundary. Specifically, we prove the existence of minimal and maximal weak solutions between an
ordered pair of (not necessarily bounded) sub- and super-solutions. The nonlinearities in both the differential equation and on the boundary are assumed to be quasimonotone and satisfy suitable growth conditions up to the critical Sobolev exponents in the appropriate Lebesgue spaces (in duality). Our proof relies on Zorn’s lemma, Kato’s inequality, and a combination of estimates and existence results for scalar nonlinear elliptic equations with growth up to critical exponents. In addition, we provide an example to illustrate our results.\\\\

% \new{Specifically, we establish the existence of minimal and maximal weak solutions between an ordered pair of sub- and supersolutions when the nonlinearities satisfy growth conditions up to the critical Sobolev exponents, using Zorn's lemma combined with Kato's inequality. As a corollary, we obtain the existence of minimal and maximal weak solutions when the nonlinearities satisfy additional monotonicity conditions.} In addition, we provide concrete examples that illustrate the applicability of our theoretical results.
\end{abstract}
\maketitle

\section{Introduction}
\label{sec:intro}
In this paper, we investigate a class of semilinear elliptic systems with nonlinearities present both in the differential equation and on the boundary condition. In particular, we consider the following problem:

\begin{equation}
	\label{pde:system:each}
	\left\{
	\begin{array}{rcll}
		-\Delta u_1 +c_1(x)u_1&=&f_1(x,u_1,u_2),   \quad &\mbox{in}\quad \Omega\,,\\
		-\Delta u_2 +c_2(x) u_2&=&f_2(x,u_1,u_2),  \quad &\mbox{in}\quad \Omega\,,\\
		\frac{\partial u_1}{\partial \eta} &=&  g_1(x,u_1,u_2)\quad &\mbox{on}\quad \partial \Omega,\\
		\frac{\partial u_2}{\partial \eta}  &=&   g_2(x,u_1,u_2)\quad &\mbox{on}\quad \partial \Omega,
	\end{array}
	\right. 
\end{equation}
where $\Omega \subset \mathbb{R}^{N}$  ($N\ge 2$)  is a bounded domain with $C^{0,1}$ boundary, denoted by  
$\p \Om$ and $\p/\p\eta
:=\eta(x)\cdot\nabla$ denotes the outer normal derivative on the boundary.

 For each $i=1,2$, the coefficients $c_i\in L^{\infty}(\Omega)$, and $f_i: \Omega \times \mathbb{R}^2 \to \mathbb{R}$  and $g_i: \partial\Omega \times \mathbb{R}^2 \to \mathbb{R}$ are Carath\'{e}odory functions;   that is, $f_i(\cdot, u_1,u_2)$ and $g_i(\cdot, u_1,u_2)$ are measurable for all $(u_1,u_2)\in \mathbb{R}^2$, and  $f_i(x, \cdot, \cdot) $ and $g_i(x, \cdot, \cdot) $ are continuous  for a.e. $x \in \Omega$ and  a.e. $x \in \partial\Omega$, respectively. Throughout this paper, we assume that each $f_i$ and $g_i$ satisfy the following quasimonotonicity condition:
\begin{enumerate}
\item[\textbf{(Q)}] $f_i$ is quasimonotone  nondecreasing in the sense that  $f_1 (x, u_1, u_2)$ is nondecreasing in $u_2$ for all fixed $x\in \Omega, u_1\in \mathbb{R}$, and  $f_2 (x, u_1, u_2)$ is nondecreasing in $u_1$ for all  fixed $x\in \Omega, u_2 \in \mathbb{R}$, and similarly for $g_i$.
\end{enumerate}

Nonlinear elliptic problems with nonlinear boundary conditions arise naturally in applications where interactions occur both within the domain and at its boundary. These problems arise in chemical reactions \cite{lacey1998}, biological processes \cite{cantrell2006}, and ecological dynamics \cite{arrieta1999}, where species interactions or chemical reactions occur in narrow boundary layers. Finding solutions becomes challenging when the nonlinearities have less regularity and solutions may be unbounded. 

The sub-supersolution method is a fundamental tool for studying such problems. This technique was originally developed for classical solutions by Amann \cite{amann1971}, Sattinger \cite{sattinger1972}, and Pao \cite{pao1992} to establish the existence of solutions between ordered pairs of barrier functions. For problems with linear boundary conditions, this method has been extensively developed and applied to obtain both classical and weak solutions, see \cite{cuesta1997} for quasilinear problems and \cite{carl2007} for a comprehensive treatment. The analysis becomes more complex when nonlinear boundary conditions are present. For classical solutions, in \cite{amann1976, pao1992} the authors established results for monotone nonlinearities via monotone iterative methods. Recent advances have focused on weak solution frameworks for nonsmooth data, see e.g. \cite{bandyopadhyay2022,mavinga2024}.

 For the case of systems,  Mitidieri and Sweers \cite{mitidieri1994} established the existence of maximal solutions for quasimonotone elliptic systems with homogeneous Dirichlet boundary conditions using monotone iteration techniques for the monotone nonlinearities case. Cuesta Leon \cite{cuesta1997} extended this framework to weak solutions of quasimonotone $p$-Laplacian systems with Dirichlet boundary conditions, addressing both monotone and nonmonotone nonlinearities cases through monotone iterations and Zorn's lemma, respectively. Recently, Bandyopadhyay et al.\ \cite{bandyopadhyay2022} studied the existence of minimal and maximal weak solutions between ordered pairs of sub and supersolutions for elliptic scalar equations with nonlinear boundary conditions. Later, Bandyopadhyay, Lewis, and Mavinga \cite{bandyopadhyay2025} extended the results to coupled systems with linear interior equations and quasimonotone nonlinear boundary conditions, treating both monotone and nonmonotone cases. In all of the above-cited papers, the authors assumed the growth conditions in the nonlinearities are bounded away from the critical Sobolev exponents (in the dual Lebesgue spaces). 
 
  In this paper, we allow the growth in the nonlinearities to go all the way to the critical Sobolev exponents (in the appropriate dual Lebesgue spaces).
For scalar equations, Mavinga, Myers, and Nkashama \cite{mavinga2025} proved the existence of minimal and maximal weak solutions for the nonmonotone nonlinearities with growth up to the critical Sobolev exponents in the appropriate Lebesgue spaces (in duality), thereby resolving a longstanding open problem in this setting. Our goal is to extend the existence results for scalar equations in \cite{mavinga2025} to coupled quasimonotone systems, allowing the nonlinearities to exhibit growth up to the critical Sobolev exponents in the appropriate dual Lebesgue spaces. Specifically, we prove the existence of minimal and maximal weak solutions between an ordered pair of (not necessarily bounded) sub- and supersolutions. Our proof relies on existence results for scalar nonlinear elliptic equations from \cite{mavinga2025}, together with a combination of Zorn's lemma, Kato's inequality and a-priori estimates. 

Throughout this paper, we will state and prove our results for the case $N>2$. The case $N=2$ follows similar arguments with appropriate adjustments.
We will use Lebesgue spaces $L^p(\Omega)$ and $L^q(\partial\Omega)$, along with their (norm) duals spaces, as well as the Sobolev space $H^1(\Omega)$ and the product space $H^1(\Omega)\times H^1(\Omega)=(H^1(\Omega))^2$ with norm $||(u_1,u_2)||_{(H^1(\Omega))^2}=||u_1||_{H^1(\Omega)}+ ||u_2||_{H^1(\Omega)}$. We also employ the Sobolev embedding theorem for $H^1(\Omega)$ into $L^p(\Omega)$, which is continuous  for $1\le p\le 2^* $ and compact for $1\le p<2^*$, where $2^*:=\dfrac{2N}{N-2}$ is 
the critical Sobolev exponent and whose conjugate exponent is $\dfrac{2N}{N+2}$. In addition, we use the boundary trace operator from $H^1(\Omega)$ into $L^q(\partial\Omega)$, which is continuous  for $1\le q\le \dfrac{2(N-1)}{N-2}$ and compact for $1\le q< \dfrac{2(N-1)}{N-2}$ and the conjugate  exponent is $\dfrac{2(N-1)}{N}.$

Notice that
\begin{equation}\label{comp1}
1<\frac{2N}{N+2}<2<2^* \quad\text{and}\quad 1<\frac{2(N-1)}N<2<\dfrac{2(N-1)}{N-2}.
\end{equation}
These inequalities, which compare the critical Sobolev exponents with their conjugates, will play a significant role throughout the {paper}.

\begin{defi}\label{defi:weak-solution}
A pair $(u_1,u_2)\in H^1(\Omega)\times H^1(\Omega)$ is a weak solution of \eqref{pde:system:each} if
\begin{enumerate}
\item[(i)]  $f_i(\cdot, u_1(\cdot), u_2(\cdot))\in L^\frac{2N}{N+2}(\Omega)$,  $g_i(\cdot, u_1(\cdot), u_2(\cdot))\in L^\frac{2(N-1)}{N}(\partial\Omega)$, for $i=1,2$ 
\vspace{0.1 in}
\item[(ii)]
$\displaystyle{\int_{\Omega} \nabla u_i\nabla \psi +\int_{\Omega} c_i(x) u_i\psi=\int_{\Omega}f_i(x,u_1,u_2)\,\psi+ \int_{\partial\Omega} g_i(x,u_1,u_2)\,\psi\,}$	
for all $\psi\in H^1(\Omega)$, for $i=1,2$ .
\end{enumerate}
\end{defi}

\begin{defi}\label{defi:sub-super-solution}
We say that $(\overline{u}_1,\overline{u}_2)\in H^1(\Omega)\times H^1(\Omega)$ is a \textit{weak supersolution} of \eqref{pde:system:each} if
\begin{enumerate}
\item[(i)]  $f_i(\cdot, u_1(\cdot), u_2(\cdot))\in L^\frac{2N}{N+2}(\Omega)$,  $g_i(\cdot, u_1(\cdot), u_2(\cdot))\in L^\frac{2(N-1)}{N}(\partial\Omega)$, for $i=1,2$ 
\vspace{0.1in}
\item[(ii)]
$\displaystyle{\int_{\Omega} \nabla \overline{u}_i\nabla \psi +\int_{\Omega} c_i(x)\overline{u}_i\psi\geq \int_{\Omega}f_i(x,\overline{u}_1,\overline{u}_2)\,\psi+ \int_{\partial\Omega} g_i(x,\overline{u}_1,\overline{u}_2)\,\psi\,}$	for all $0\leq \psi\in H^1(\Omega)$, for $i=1,2$ .
\end{enumerate}
\end{defi}
A \textit{weak subsolution} $(\underline{u}_1,\underline{u}_2)\in H^1(\Omega)\times H^1(\Omega)$ is defined by reversing inequality (ii) above. 

For our purposes, we consider the component-wise ordering, that is,
 for  $(u_1, u_2) ,  (v_1,v_2) \in $  $H^1(\Omega)\times H^1(\Omega)$ 
\begin{equation}\label{ord1}
    (u_1, u_2) \le  (v_1,v_2)\;  \text{if}\;  u_i\le v_i, \ i=1,2.
\end{equation}

In what follows, we state our main result regarding the existence of weak solutions under the assumption \textbf{(Q)}.

\begin{thm}\label{nonmonotone} Suppose there exist a weak subsolution $(\underline{u}_1, \underline{u}_2)$ and supersolution $(\overline{u}_1, \overline{u}_2)$ of \eqref{pde:system:each} with 
$(\underline{u}_1, \underline{u}_2) \le (\overline{u}_1, \overline{u}_2),$ and the functions $f_i$ and $ g_i$ ($i=1,2$) satisfy the following growth conditions:

\begin{enumerate}
    \item[\textbf{(H1)}] There exist $K_1, K_2 \in L^{\frac{2N}{N+2}}(\Omega)$, such that 
    $$|f_i(x,u_1,u_2)| \le K_i(x) \quad \text{a.e. } x \in \Omega,$$
    whenever $\underline{u}_i(x) \le u_i \le \overline{u}_i(x)$, $i=1,2$.
    \vskip0.2cm
    \item[\textbf{(H2)}] There exist $\widetilde{K}_1, \widetilde{K}_2 \in L^{\frac{2(N-1)}{N}}(\partial\Omega)$,  such that 
    $$|g_i(x,u_1,u_2)| \le \widetilde{K}_i(x) \quad \text{a.e. } x \in \partial\Omega,$$
    whenever $\underline{u}_i(x) \le u_i \le \overline{u}_i(x)$, $i=1,2$.
\end{enumerate}

Then the following holds:

\begin{enumerate}
    \item[(i)] System \eqref{pde:system:each} has a weak solution $(u_1, u_2)$ satisfying 
    \begin{equation}
    \label{eq:order}
      (\underline{u}_1, \underline{u}_2) \le (u_1, u_2) \le (\overline{u}_1, \overline{u}_2).  
    \end{equation}

    \item[(ii)] System \eqref{pde:system:each} has a minimal weak solution $(u_{1,*}, u_{2,*})$ and a maximal weak solution $(u_1^*, u_2^*)$ satisfying \eqref{eq:order}. Moreover, for any weak solution $(u_1, u_2)$ of \eqref{pde:system:each} with $(\underline{u}_1, \underline{u}_2) \le (u_1, u_2) \le (\overline{u}_1, \overline{u}_2)$, we have
    \begin{equation}
    (u_{1,*}, u_{2,*}) \le (u_1, u_2) \le (u_1^*, u_2^*).
    \end{equation}
\end{enumerate}
\end{thm}

\begin{rem}\label{monotone}
If additional monotonicity conditions are imposed on the nonlinearities $f_i$ and $g_i$, the existence of weak solutions can be obtained via a constructive monotone 
iteration scheme, which can be useful for numerical implementation. Precisely, 

\smallskip 
 suppose there exists a weak subsolution $(\underline{u}_1, \underline{u}_2)$ and supersolution $(\overline{u}_1, \overline{u}_2)$ of \eqref{pde:system:each} with 
$(\underline{u}_1, \underline{u}_2) \le (\overline{u}_1, \overline{u}_2)$ and the functions $f_i, g_i$ ($i=1,2$) satisfy the following monotonicity conditions:

\begin{enumerate}
    \item[\textbf{(A1)}] There exist constants $\widehat{k}_i \geq 0$ such that for all $x \in \Omega$, the function $f_i(x, s_1, s_2) + \widehat{k}_i s_i$ is non-decreasing for $\underline{u}_i\le s_i \le \overline{u}_i$, $i = 1, 2$.
    
    \item[\textbf{(A2)}] There exist constants $\overline{k}_i \geq 0$ such that for all $x \in \partial\Omega$, the function $g_i(x, s_1, s_2) + \overline{k}_i s_i$ is non-decreasing for $\underline{u}_i\le s_i \le \overline{u}_i$, $i = 1, 2$.
\end{enumerate}
then system \eqref{pde:system:each} has a minimal weak solution $(u_{1,*}, u_{2,*})$ and a maximal weak solution $(u_1^*, u_2^*)$,  obtained as limits of the monotone sequences \begin{equation*}
(\underline{u}_1,\underline{u}_2) \leq (u_{1,1},u_{2,1}) \leq \cdots \leq 
(u_{1,n},u_{2,n}) \leq \cdots \leq (w_{1,n},w_{2,n}) \leq \cdots \leq 
(w_{1,1},w_{2,1}) \leq (\overline{u}_1,\overline{u}_2),
\end{equation*}
that is, $(u_{1,n},u_{2,n})\to (u_{1,*}, u_{2,*})$ and $(w_{1,n},w_{2,n}) \to (u_1^*, u_2^*)$ as $n\to \infty$, and 
\begin{align*}
 (\underline{u}_1, \underline{u}_2) \le (u_{1,*}, u_{2,*}) \le (u_1^*, u_2^*)  \le (\overline{u}_1, \overline{u}_2) 
 % \;\mbox{ and} \quad
 % (\underline{u}_1, \underline{u}_2) \le (u_1^*, u_2^*)  \le (\overline{u}_1, \overline{u}_2).
\end{align*}
\end{rem}

The remainder of this paper is organized as follows. In Section \ref{sec:prelim}, we present preliminary results including Kato's inequality for the single equation case and  the existence result for scalar equations. Section \ref{sec:nonmonotone} deals with the proof of Theorem \ref{nonmonotone}. In Section \ref{sec:examples}, we present an application of our main result. Finally, the appendix contains essential results related to Kato-type inequality for systems, which we used in the proof of Theorem \ref{nonmonotone}.

\section{Preliminaries}
\label{sec:prelim}
In this section, we state  the existence results for  scalar equations from \cite{mavinga2025}. For completeness, we also establish Kato's inequality for single equations, showing that the maximum of two subsolutions remains a subsolution and the minimum of two supersolutions remains a supersolution.

\begin{prop}\label{single} [Existence result for a scalar elliptic problem]
Consider the nonlinear problem
\begin{align}\label{scalar}
-\Delta u + c(x)u &= f(x, u) \quad \text{in } \Omega, \nonumber\\
\frac{\partial u}{\partial \eta} &= g(x, u) \quad \text{on } \partial\Omega,
\end{align}
where $c \in L^\infty(\Omega)$, and $f: \Omega \times \mathbb{R} \to \mathbb{R}$ and $g: \partial\Omega \times \mathbb{R} \to \mathbb{R}$ are Carathéodory functions. Suppose that there exists a pair of  weak sub and supersoultion, $\underline{u}$, $\overline{u}$ such that $\underline{u} \leq \overline{u}$ in $\Omega$, and  there exists $K_1 \in L^{\frac{2N}{N+2}}(\Omega)$ and  $K_2 \in L^{\frac{2(N-1)}{N}}(\partial\Omega)$, such that $|f(x, s)| \leq K_1(x)$ a.e. $x \in \Omega$ and $|g(x, s)| \leq K_2(x)$ a.e. $x \in \partial\Omega$ for all $s \in \mathbb{R}$ satisfying $\underline{u}(x) \leq s \leq \overline{u}(x)$. Then there exists at least one weak solution $u$ of eq. \eqref{scalar} such that $\underline{u} \leq u \leq \overline{u}$.
\end{prop}

\begin{proof}
See \cite{mavinga2025} for the proof. 
\end{proof}
In the next proposition, we state Kato's inequality up to the boundary and provide a proof for completeness, since it was not included in [12], \cite[Prop.\ 2.3]{mavinga2025}.
\begin{prop}\label{prop:main_kato} [Kato's inequality for single equations]
Let $c \in L^\infty(\Omega)$. Let $u_1$ and $u_2$   be  functions in $H^1(\Om)$ such that there exist  $\widetilde{f_1}, \widetilde{f_2}\in L^{\frac{2N}{N+2}}(\Omega)$, $\widetilde{g_1}, \widetilde{g_2}\in L^{\frac{2(N-1)}{N}}(\partial\Omega)$, satisfying
\begin{equation}\label{linear:i} 
\int_{\Om}\left(\nabla u_i \nabla \psi + c(x) u_i\psi\right) \leq \int_{\Om}\widetilde{f_i}\psi+\int_{\p \Om}\widetilde{g_i}\psi\qquad \text{for all }\ 0 \le \psi \in H^1(\Om)\,,  
\end{equation}
for $i=1,2.$
Then, $u:=\max\{u_1,u_2\}$ satisfies
\begin{equation*}\label{weak:subsol}
\int_{\Om}\left(\nabla u \nabla \psi + c(x) u\psi\right) \leq \int_{\Om}f\psi+\int_{\p \Om}g\psi\qquad \text{for all }\ 0 \le \psi \in H^1(\Om)\,,
\end{equation*}
where 
$f(x):=\begin{cases}
\widetilde{f_1}(x) &\quad \text{if} \ u_1(x)>u_2(x)\\
\widetilde{f_2}(x) &\quad \text{if} \ u_1(x)\le u_2(x),
\end{cases}$
a.e. $x\in \Om $ 

and 

$g(x):=\begin{cases}
\widetilde{g_1}(x) &\quad \text{if} \ u_1(x)>u_2(x)\\
\widetilde{g_2}(x) &\quad \text{if} \ u_1(x)\le u_2(x),
\end{cases}$
a.e. $x\in \partial\Om $. 
\end{prop}
\begin{proof}
Define 
$$
\Om_1:=\{x\in \Om: u_1(x)>u_2(x)\}\;  {\text{and}}\ 
\Om_2:=\Om \setminus \Om_1\,$$
and $$
\Gamma_1:=\{x\in \partial\Om: u_1(x)>u_2(x)\}\; {\text{and}}\; \Gamma_2:=\partial\Om\setminus \Gamma_1\,.
$$
Fix $0 \le \psi \in H^1(\Om)$.
Then, 
\begin{equation*} 
\label{eq:I}
\begin{split}
I & = \int_{\Om}\nabla u \nabla \psi + \int_{\Om}c(x) u\psi \\
& =\underbrace{\int_{\Om_1}\left(\nabla u_1 \nabla \psi + c(x)u_1\psi\right)}_{I_1}+ \underbrace{\int_{\Om_2}\left(\nabla u_2 \nabla \psi + c(x) u_2\psi\right)}_{I_2}\,.
\end{split}
\end{equation*}
Consider a sequence $\xi_n\in C^1(\mathbb{R})$ such that  $$\xi_n(t):=\begin{cases}
1 \quad \text{if} \ t\ge 1/n \\
0 \quad \text{if} \  t\le 0,
\end{cases}$$  and $\xi_n'>0$ on $(0,1/n).$
Then, define the sequence of functions
$$r_n(x):= \xi_n((u_1-u_2)(x)) \quad \quad \text{for}\;x\in \overline{\Om}\,.$$
Observe that $r_n\in H^1(\Om)$ and $r_n$ converges point-wise to $\displaystyle \chi_{\Om_1\cup \Gamma_1}$,  where  the characteristic function is defined as $\displaystyle \chi_{\Om_1\cup \Gamma_1}(x):=\begin{cases}
1 \quad \text{if} \ x\in \Om_1\cup\Gamma_1\; \\
0 \quad \text{if} \  \text{otherwise}\,.
\end{cases}$
Moreover, $||r_n||_{L^{\infty}(\Om)\cap {L^{\infty}(\partial\Om)}}\le 1$  and 
$\text{supp}(\nabla r_n)\subset \overline{D}_n$, where $D_n:=\{x\in \Om: 0<u_1(x)-u_2(x)<\frac{1}{n}\}$.
Then, using the Lebesgue dominated convergence theorem,  we have that
$$
\displaystyle I_1= \lim_{n\to \infty}\left[\int_{\Om}r_n\nabla u_1 \nabla \psi + \int_{\Om}r_nc(x)u_1\psi\right].
$$
Since $r_n\in H^1(\Om)\cap L^{\infty}(\Om)\cap {L^{\infty}(\partial\Om)}$, it follows that $r_n\psi\in H^1(\Om)$ for any test function $\psi\in H^1(\Om)\cap L^{\infty}(\Om)$.  Recalling that  $\nabla r_n=0$ on $\Om\setminus D_n$, and that $u_1$ satisfies \eqref{linear:i}, we can write
\begin{align}
 \int_{\Om}r_n\nabla u_1 \nabla \psi + r_nc(x)u_1\psi
&=\int_{\Om}\nabla u_1 \nabla (r_n \psi)+ c(x)u_1(r_n\psi) - \int_{D_n}\psi\nabla u_1 \nabla r_n \nonumber \\
&\le \int_{\Om}f_1 (r_n\psi) + \int_{\partial\Om}g_1 (r_n\psi) - \int_{D_n}\psi\nabla u_1 \nabla r_n \label{inew:f1}\,.
\end{align} 
Taking the limit as $n\to \infty$  for the first and second terms on the right-hand side of \eqref{inew:f1}, using the Lebesgue dominated convergence theorem, we get 
$$
\lim\limits_{n \to \infty}\int_{\Om}f_1 r_n\psi  = \int_{\Omega_1}f_1 \psi\; \mbox{ and }\; \lim\limits_{n \to \infty}\int_{\partial\Om}g_1 r_n\psi  = \int_{\Gamma_1}g_1 \psi\,.
$$
Likewise, for $I_2$ we have
$$
\displaystyle I_2= \lim_{n\to \infty}\left[\int_{\Om}(1-r_n)\nabla u_2 \nabla \psi + \int_{\Om}(1-r_n)c(x)u_2\psi\right],
$$
and
\begin{align}
&\int_{\Om}(1-r_n)\nabla u_2 \nabla \psi + \int_{\Om}(1-r_n)c(x)u_2\psi \nonumber\\ &\qquad =\int_{\Om} \nabla u_2 \nabla [(1-r_n )\psi]+ c(x)u_2(1-r_n)\psi + \int_{D_n}\psi\nabla u_2 \nabla r_n \nonumber\\
&\qquad\le \int_{\Om}f_2 (1-r_n)\psi+\int_{\partial\Om}g_2 (1-r_n)\psi + \int_{D_n}\psi\nabla u_2 \nabla r_n.
\label{inew:f2}
\end{align}
Taking the limit as $n\to \infty$  for the first and second terms on the right-hand side of  \eqref{inew:f2} and using the Lebesgue dominated convergence theorem, we get
$$
\lim\limits_{n \to \infty}\int_{\Om}f_2 (1-r_n)\psi  = \int_{\Omega_2}f_2 \psi\;\mbox{ and }\;\lim\limits_{n \to \infty}\int_{\partial\Om}g_2 (1-r_n)\psi  = \int_{\Gamma_2}g_2 \psi\,.
$$
Using the fact that $\nabla r_n=\xi_n'(u_1-u_2)\nabla(u_1-u_2)$, the sum of the second terms of the right-hand side of  \eqref{inew:f1} and  \eqref{inew:f2} yields 
\begin{align}
-\int_{D_n}\psi\nabla u_1 \nabla r_n + \int_{D_n}\psi\nabla u_2 \nabla r_n &= - \int_{D_n}\psi\nabla (u_1-u_2) \nabla r_n \nonumber \\
& = -  \int_{D_n}\psi\xi_n'(u_1-u_2) |\nabla (u_1-u_2) |^2\le 0,
\label{omega:n}
\end{align}
since $\psi\ge 0\; {\text{and}}\;\xi_n'\ge 0$.
Adding \eqref{inew:f1} and \eqref{inew:f2},  taking the limit, and using \eqref{omega:n}, we get 
\begin{equation*}
I=I_1+I_2\le  \int_{\Omega_1}f_1 \psi + \int_{\Omega_2}f_2 \psi+\int_{\Gamma_1}g_1 \psi + \int_{\Gamma_2}g_2 \psi= \int_{\Om} f\psi+\int_{\partial\Om} g\psi\,.
\end{equation*}
Thus, $u:=\max\{u_1, u_2\}$ satisfies 
\begin{equation*}
\int_{\Om}\left(\nabla u \nabla \psi + c(x)u\psi\right) \leq  \int_{ \Om}f\psi+\int_{\p \Om}g\psi\qquad \text{for all }\ 0 \le \psi \in H^1(\Om)\,, 
\end{equation*}
completing the proof of Proposition~\ref{prop:main_kato}.
\end{proof}

Likewise, we have a result for the minimum of two supersolutions.
\begin{cor}
\label{cor}
Let $c \in L^\infty(\Omega)$. Let $u_1$ and $u_2$   be  functions in $H^1(\Om)$ such that there exist $\widetilde{f_1}, \widetilde{f_2}\in L^{\frac{2N}{N+2}}(\Omega)$, $\widetilde{g_1}, \widetilde{g_2}\in L^{\frac{2(N-1)}{N}}(\partial\Omega)$ satisfying
\begin{equation*}\label{linear:imin} 
\int_{\Om}\left(\nabla u_i \nabla \psi + c(x)u_i\psi\right) \ge \int_{ \Om}\widetilde{f_i}\psi+  \int_{\p \Om}\widetilde{g_i}\psi\qquad \text{for all }\ 0\le \psi \in H^1(\Om)\,, 
\end{equation*}
for $i=1,2.$
Then, $u:=\min\{u_1,u_2\}$ satisfies 
\begin{equation*}\label{linear:umin} 
\int_{\Om}\left(\nabla u \nabla \psi + c(x)u\psi\right) \ge \int_{ \Om}f\psi+\int_{\p \Om}g\psi,\qquad \text{for all }\ 0 \le \psi \in H^1(\Om)\,, 
\end{equation*}
where $f(x):=\begin{cases}
\widetilde{f_1}(x) &\quad \text{if} \ u_1(x)>u_2(x)\\
\widetilde{f_2}(x) &\quad \text{if} \ u_1(x)\le u_2(x),
\end{cases}$
\quad a.e. $x\in \Om$,

and 

$g(x):=\begin{cases}
\widetilde{g_1}(x) &\quad \text{if} \ u_1(x)>u_2(x)\\
\widetilde{g_2}(x) &\quad \text{if} \ u_1(x)\le u_2(x),
\end{cases}$
\quad a.e. $x\in \partial\Om $. 
\end{cor}
\begin{proof}
Using the fact that  $\min\{u_1, u_2\} = -\max\{-u_1, -u_2\}$, the proof follows from Proposition~ \ref{prop:main_kato}.
\end{proof}

\section{Proof of Theorem \ref{nonmonotone}}
\label{sec:nonmonotone}
In this section, we present the proof of Theorem \ref{nonmonotone} using the existence result for a scalar elliptic problem (Proposition \eqref{single}), and  Zorn's lemma applied to an appropriately defined set of subsolutions.  Consider the set
\begin{equation}
\label{def:J}
J=\{(u_1,u_2)\in (H^1(\Omega))^2  \colon (\underline{u}_1,\underline{u}_2)\leq (u_1,u_2)\leq (\overline{u}_1,\overline{u}_2),(u_1,u_2) \mbox{ is a subsolution of \eqref{pde:system:each}} \}.
\end{equation}
Let $(\widetilde{u}_1, \widetilde{u}_2) \in J$ be fixed. We show that there exists a subsolution of \eqref{pde:system:each}, $(w_1,w_2) \in J$, satisfying $(\widetilde{u}_1, \widetilde{u}_2) \le (w_1,w_2) \le (\overline{u}_1, \overline{u}_2)$ with $\|(w_1,w_2)\|_{(H^1(\Omega))^2} \le C$, where $C$ does not depend on $(w_1,w_2)$.

Indeed, fix $\widetilde{u}_2$  and consider the scalar equation
\begin{equation}\label{fv}
\left\{
	\begin{array}{rcll}
    -\Delta u_1 + c_1(x)u_1 + ku_1 &=& f_1(x,u_1,\widetilde{u}_2) + ku_1 & \text{in} \; \Omega;\\
    \dfrac{\partial u_1}{\partial \eta} &=& g_1(x,u_1,\widetilde{u}_2) & \text{on}\; \partial\Omega.
    \end{array}
	\right.
\end{equation}
Define $\overline{f}_1(x,s_1) := f_1(x,s_1,\widetilde{u}_2(x)) + ks_1$ and
$\overline{g}_1(x,s_1) := g_1(x,s_1,\widetilde{u}_2(x))$, where
$$k > \max\{\|c_1\|_{L^\infty(\Omega)}, \|c_2\|_{L^\infty(\Omega)}\}$$ is fixed
so that $c_i(x) + k > 0$ a.e.\ in $\Omega$ and $c_i + k \in L^\infty(\Omega)$ for $i = 1, 2$. 

Note that the norm $\|u\|_{c_i+k}^2 := \int_\Omega |\nabla u|^2 + \int_\Omega (c_i(x) + k) u^2$
is equivalent to the standard $H^1(\Omega)$ norm. For notational simplicity, we
denote $\|\cdot\|_{c_i+k}$ by $\|\cdot\|_{H^1(\Omega)}$ in what follows.
Since $(\widetilde{u}_1, \widetilde{u}_2) \in J$, we have
$\underline{u}_i \leq \widetilde{u}_i \leq \overline{u}_i$ for $i = 1,2$.
Therefore, for any $s_1$ satisfying $\underline{u}_1(x) \leq s_1 \leq \overline{u}_1(x)$,
it follows from \textbf{(H1)}  that
$$|\overline{f}_1(x,s_1)| \leq |f_1(x,s_1,\widetilde{u}_2(x))| + k|s_1| \leq K_1(x) + k\left(|\underline{u}_1(x)| + |\overline{u}_1(x)|\right).$$
Since $\underline{u}_1, \overline{u}_1 \in H^1(\Omega)$, the Sobolev embedding from $ H^1(\Omega)$ into $L^p(\Omega)$ together with \eqref{comp1} gives  $\underline{u}_1$ and $ \overline{u}_1$  are in $L^{\frac{2N}{N+2}}(\Omega)$.
%\hookrightarrow L^{2^*}(\Omega)
%\hookrightarrow L^{2_*}(\Omega)$ (see \eqref{comp1}),

Hence, 
\begin{equation}
\label{k-bound}
|\overline{f}_1(\cdot, s_1)| \leq \overline{K}_1(\cdot)\in L^{\frac{2N}{N+2}}(\Omega), 
\end{equation}
where $\overline{K}_1(\cdot):=K_1(\cdot) + k\left(|\underline{u}_1(\cdot)|
+ |\overline{u}_1(\cdot)|\right).$ 
Similarly, from \textbf{(H2)}, we have that \begin{equation}
\label{k2-bound}|\overline{g}_1(.,s_1)| = |g_1(.,s_1,\widetilde{u}_2(.))|
\leq \widetilde{K}_1(.) \in L^{\frac{2(N-1)}{N}}(\partial\Omega).
\end{equation}
By Proposition~\ref{single}, it follows that there exists
a solution of \eqref{fv},  $w_1 \in H^1(\Omega)$, such that
\begin{equation}\label{w1:subsolution}
    \widetilde{u}_1 \leq w_1 \leq \overline{u}_1.
\end{equation}
%$\widetilde{u}_1 \leq w_1 \leq \overline{u}_1$. 
Observe that  $w_1 \in H^1(\Omega)$ is a solution of \eqref{fv} if and only if  $w_1$ is 
also a solution of
\begin{equation}\label{fv'}
\left\{
	\begin{array}{rcll}
    -\Delta w_1 + c_1(x)w_1 &=& f_1(x,w_1,\widetilde{u}_2) & \text{in} \; \Omega;\\
    \dfrac{\partial w_1}{\partial \eta} &=& g_1(x,w_1,\widetilde{u}_2) & \text{on}\; \partial\Omega,
    \end{array}
	\right.
\end{equation}
%since the term $kw_1$ cancels from both sides of \eqref{fv}.

Similarly, with $\widetilde{u}_1$ fixed, consider the scalar equation
\begin{equation}\label{sv}
\left\{
	\begin{array}{rcll}
    -\Delta u_2 + c_2(x)u_2 + ku_2 &=& f_2(x,\widetilde{u}_1,u_2) + ku_2 & \text{in} \; \Omega;\\
    \dfrac{\partial u_2}{\partial \eta} &=& g_2(x,\widetilde{u}_1,u_2) & \text{on}\; \partial\Omega,
    \end{array}
	\right.
\end{equation}
where $\overline{f}_2(x,s_2) := f_2(x,\widetilde{u}_1(x),s_2) + ks_2$ and
$\overline{g}_2(x,s_2) := g_2(x,\widetilde{u}_1(x),s_2)$. By an analogous argument,
% using conditions \textbf{(H1)} and \textbf{(H2)} and
% Proposition~\ref{single} are satisfied, 
there exists a solution of \eqref{sv},
 $w_2 \in H^1(\Omega)$, such that 
 \begin{equation}\label{w2:subsolution}
 \widetilde{u}_2 \leq w_2 \leq \overline{u}_2.
 \end{equation}
Note that  $w_2$ is a solution of \eqref{sv} if and only if  $w_2$ is also a solution of
\begin{equation}\label{sv'}
\left\{
	\begin{array}{rcll}
    -\Delta w_2 + c_2(x)w_2 &=& f_2(x,\widetilde{u}_1,w_2) & \text{in} \; \Omega;\\
    \dfrac{\partial w_2}{\partial \eta} &=& g_2(x,\widetilde{u}_1,w_2) & \text{on}\; \partial\Omega.
    \end{array}
	\right.
\end{equation}
%since the term $kw_2$ cancels from both sides of \eqref{sv}.

Using condition {\bf (Q)}, \eqref{w1:subsolution} and \eqref{w2:subsolution}, it follows that $$f_1(x,w_1,\widetilde{u}_2) \le f_1(x,w_1,w_2) \mbox{ and }g_1(x,w_1,\widetilde{u}_2) \le g_1(x,w_1,w_2).$$ Similarly, $$f_2(x,\widetilde{u}_1,w_2) \le f_2(x,w_1,w_2) \mbox{ and } g_2(x,\widetilde{u}_1,w_2) \le g_2(x,w_1,w_2).$$ Therefore,
\begin{equation}
\label{eq:w-sub1}
 \int_{\Omega}( \nabla w_i\nabla \psi + c_i(x)w_i\psi)\le  \int_{\Omega} f_i(x,w_1,w_2)\,\psi\, + \int_{\partial\Omega} g_i(x,w_1,w_2)\,\psi\, \quad \text{for all } \; 0\le \psi \in H^1(\Omega).   
\end{equation}
Moreover, since $\underline{u}_i \le \widetilde{u}_i \le w_i \le \overline{u}_i$, (\textbf{H1}) and (\textbf{H2}) together imply that
\begin{equation}
    \label{eq:w-sub2}
 f_i(\cdot, w_1(\cdot), w_2(\cdot))\in L^\frac{2N}{N+2}(\Omega) \mbox{ and } g_i(\cdot, w_1(\cdot), w_2(\cdot))\in L^\frac{2(N-1)}{N}(\partial\Omega) \mbox{ for } i=1,2.   
\end{equation}
From \eqref{eq:w-sub1} and \eqref{eq:w-sub2}, it follows that  $(w_1,w_2)$ is a subsolution of \eqref{pde:system:each}.
%??By the equivalence of norms, see \eqref{eq:equivalent-norms}, there exists a constant $C_0>0$ such that $\|w_1\|^2_{H^1(\Omega)} \le C_0\|w_1\|^2_{c_1}$. 

Thereafter, taking $\psi = w_1$ as a test function in the weak formulation of \eqref{fv}, 
and {using} \eqref{k-bound}-\eqref{k2-bound} we have
\begin{align}
\|w_1\|^2_{H^1(\Omega)} & = \int_\Omega|\nabla w_1|^2 + \int_\Omega(c_1(x)+k)w_1^2 
\notag\\
&= \int_\Omega\overline{f}_1(x,w_1)w_1 
+ \int_{\partial\Omega}\overline{g}_1(x,w_1,\widetilde{u}_2)w_1 \notag\\
&\leq \int_\Omega \overline{K}_1(x)w_1 
+ \int_{\partial\Omega}\widetilde{K}_1(x)w_1 \notag\\
&\leq \int_\Omega \overline{K}_1(x)\overline{u}_1 
+ \int_{\partial\Omega}\widetilde{K}_1(x)\overline{u}_1 \notag\\
&\leq \|\overline{K}_1\|_{L^{\frac{2N}{N+2}}(\Omega)}
\|\overline{u}_1\|_{L^{\frac{2N}{N-2}}(\Omega)} 
+ \|\widetilde{K}_1\|_{L^{\frac{2(N-1)}{N}}(\partial\Omega)}
\|\overline{u}_1\|_{L^{\frac{2(N-1)}{N-2}}(\partial\Omega)} \notag \\&\le  C_1 \label{eq:bound},
\end{align}
where $C_1$ is a constant depending on $(\overline{u}_i, \underline{u}_i, \Omega, 
\overline{K}_1, \widetilde{K}_1)$ but independent of $w_1$.

Similarly, $$\|w_2\|_{H^1(\Omega)} \le C_2.$$ Hence, $\|(w_1,w_2)\|_{(H^1(\Omega))^2}\le M$, where $M$ is a constant independent of $(w_1,w_2)$. 

Thereafter, define the set of subsolutions of \eqref{pde:system:each}, $\mathcal{A}$ by
$$\mathcal{A}:=\left\{ (w_1,w_2) \in (H^1(\Omega))^2 \mid (\underline{u}_1, \underline{u}_2) \le (\widetilde{u}_1, \widetilde{u}_2) \le (w_1,w_2) \le (\overline{u}_1, \overline{u}_2) \text{ for some } (\widetilde{u}_1, \widetilde{u}_2)\in J\right\},$$
where $w_1$ and $w_2$ are solutions of \eqref{fv'} and \eqref{sv'} for $\widetilde{u}_1$ and $\widetilde{u}_2$, respectively. 

To establish the existence of a maximal element of $\mathcal{A}$, we employ Zorn's Lemma on $\mathcal{A}$ (see, e.g. \cite{halmos1960naive})  as follows:

\smallskip

\noindent {\bf Step 1:} From the existence result we have that $\mathcal{A}\neq \emptyset$. Since the product space $(H^1(\Omega))^2$ is separable, we can define the chain $Y=\{ (w_{1,n},w_{2,n})\}_{n\ge 1}$ in $(\mathcal{A}, \le)$, where we assume without loss of generality (by possible reordering) that $\{(w_{1,n},w_{2,n})\}_{n\ge 1}$ is a component-wise increasing sequence. We show that $Y$ has an upper bound in $\mathcal{A}$.
Since $(w_{1,n},w_{2,n}) \in \mathcal{A}$ for every $n$, there exists a subsolution $(\widetilde{u}_{1,n},\widetilde{u}_{2,n})$ of \eqref{pde:system:each} such that
$$(\underline{u}_1, \underline{u}_2) \le (\widetilde{u}_{1,n}, \widetilde{u}_{2,n}) \le (w_{1,n},w_{2,n}) \le (\overline{u}_1, \overline{u}_2),$$
with $w_{1,n}$ and $w_{2,n}$ solutions of \eqref{fv'} and \eqref{sv'}, respectively, for the pair $(\widetilde{u}_{1,n}, \widetilde{u}_{2,n})$. From \eqref{eq:bound}, it follows that $\|(w_{1,n},w_{2,n})\|^2_{(H^1(\Omega))^2}\le M$, where $M$
%=M(\overline{u}_i, \underline{u}_i, \Omega, K_i)$ {\color{red}????} 
is a constant independent of $n$. 

By the reflexivity of $(H^1(\Omega))^2$, there is a subsequence (relabeled), $(w_{1,n},w_{2,n})$, which converges weakly to $(w_1^*, w_2^*)\in (H^1(\Omega))^2$. Since the sequence $\{(w_{1,n},w_{2,n})\}$ is component-wise monotonically increasing and bounded above, we get 
\begin{equation*}
    \displaystyle \lim_{n \to \infty} w_{i,n}(x)=w_i^*(x) \quad \mbox{ a.e. } x\in \overline{\Omega}, \quad  i =1,2.
\end{equation*}
%$\{(w_{1,n},w_{2,n})\}$ converges pointwise to $(w_1^*, w_2^*)$. 
Using the continuity of $f_i(x,\cdot,\cdot)$ and $g_i(x,\cdot,\cdot)$, we obtain that
\begin{align*}
\lim_{n\to\infty} f_i(x,w_{1,n}(x),w_{2,n}(x)) &=f_i(x,w_1^*(x),w_2^*(x)), \quad\mbox{a.e. }x\in \Omega, \quad i =1,2.  \\
\lim_{n\to\infty} g_i(x,w_{1,n}(x),w_{2,n}(x)) &=g_i(x,w_1^*(x),w_2^*(x)), \quad\mbox{a.e. }x\in \partial\Omega, \quad i =1,2.    
\end{align*}

From conditions {\bf(H1)} and {\bf (H2)},
%and the fact that $ (\underline{u}_1, \underline{u}_2), (\underline{u}_1, \underline{u}_2)\in (H^1(\Omega))^2$, 
we get that $$|f_i(x,w_{1,n}, w_{2,n})| \le K_i(x), \quad K_i\in L^{\frac{2N}{N+2}}(\Omega),$$ and $$|g_i(x,w_{1,n}, w_{2,n})| \le \widetilde{K}_i(x), \quad \widetilde{K}_i\in L^{\frac{2(N-1)}{N}}(\partial\Omega).$$

Since $(w_{1,n},w_{2,n})\rightharpoonup(w_1^*, w_2^*)\in (H^1(\Omega))^2$, we have
\begin{equation*}
   \lim_{n\rightarrow \infty}\int_{\Omega}\nabla w_{i,n}\nabla\psi +c_i(x)w_{i,n}\psi= \int_{\Omega}\nabla w_i^*\nabla\psi + c_i(x)w_i^*\psi \quad \forall \psi \in H^1(\Omega).
\end{equation*}

Moreover, using the quasimonotonicity of $f_i$ and $g_i$, and the Lebesgue dominated convergence theorem, we have
\begin{align*}
 &\int_{\Omega}(\nabla w_1^*\nabla\psi + c_1(x)w_1^*\psi) \\
 &= \lim_{n\to\infty} \int_{\Omega}(\nabla w_{1,n}\nabla\psi + c_1(x)w_{1,n}\psi)  \\
 &= \lim_{n\to\infty} \left[\int_{\Omega}\bigl(f_1(x,w_{1,n}(x),\widetilde{u}_{2,n}(x)) w_{1,n}\bigr)\psi + \int_{\partial\Omega}g_1(x,w_{1,n}(x),\widetilde{u}_{2,n}(x))\psi\right]  \\
 &\le \lim_{n\to\infty} \left[\int_{\Omega}\bigl(f_1(x,w_{1,n}(x),w_{2,n}(x))\bigr)\psi + \int_{\partial\Omega}g_1(x,w_{1,n}(x),w_{2,n}(x))\psi\right] \\
&=\int_{\Omega}\bigl(f_1(x,w_{1}^*(x),w_{2}^*(x))\bigr)\psi + \int_{\partial\Omega}g_1(x,w_{1}^*(x),w_{2}^*(x))\psi.
\end{align*}

Similarly,
$$\int_{\Omega}(\nabla w_2^*\nabla\psi + c_2(x)w_2^*\psi)\le\int_{\Omega}\bigl(f_2(x,w_{1}^*(x),w_{2}^*(x))\bigr)\psi + \int_{\partial\Omega}g_2(x,w_{1}^*(x),w_{2}^*(x))\psi,$$
for all $0\le \psi \in H^1(\Omega)$. 
Thus, $(w_1^*, w_2^*)$ is a subsolution of \eqref{pde:system:each}.

\smallskip

\noindent{\bf Step 2:} We aim to prove existence of maximal element of $\mathcal{A}$. Indeed, let us take $(\widetilde{u}_1, \widetilde{u}_2)=(w_{1}^*, w_{2}^*)$. Then we can find a subsolution $(v_1,v_2)$ of \eqref{pde:system:each} such that $(\underline{u}_1,\underline{u}_2) \le (w_{1}^*,w_{2}^*)\le (v_1,v_2) \le (\overline{u}_1,\overline{u}_2)$, with $v_1$ and $v_2$ solutions of \eqref{fv'} and \eqref{sv'}, respectively, for the pair $(w_{1}^*, w_{2}^*)$. This implies that $(v_1,v_2)\in \mathcal{A}$ and is an upper bound of the set $Y$. By Zorn's lemma, $\mathcal{A}$ has a maximal element, $(z_1,z_2)\in\mathcal{A}$. Hence, $(z_1,z_2)$ is a subsolution of \eqref{pde:system:each}. Observe that $(z_1,z_2)\in\mathcal{A}$ implies that there exists %a subsolution of \eqref{pde:system:each}, 
$(\widetilde{z}_1,\widetilde{z}_2) \in J$, satisfying
$$(\underline{u}_1,\underline{u}_2) \le (\widetilde{z}_1,\widetilde{z}_2)\le (z_1,z_2) \le (\overline{u}_1,\overline{u}_2),$$
where $z_1$ and $z_2$ are solutions of \eqref{fv'} and \eqref{sv'}, respectively, for the pair $(\widetilde{z}_1,\widetilde{z}_2)$. Since $f_i$ and $g_i$ are quasimonotone for $i=1,2$, it follows that
\begin{align*}
    \int_{\Omega}(\nabla z_1\nabla\psi + c_1z_1\psi)&= \int_{\Omega}f_1(x,z_1,\widetilde{z}_2)\psi + \int_{\partial\Omega}g_1(x,z_1,\widetilde{z}_2)\psi \\
    &\le \int_{\Omega}f_1(x,z_1,z_2)\psi+ \int_{\partial\Omega}g_1(x,z_1,z_2)\psi,
\end{align*}
and similarly
$$\int_{\Omega}(\nabla z_2\nabla\psi + c_2z_2\psi) \le \int_{\Omega}f_2(x,z_1,z_2)\psi+ \int_{\partial\Omega}g_2(x,z_1,z_2)\psi$$
for all $0\le \psi \in H^1(\Omega)$. Therefore, $(z_1,z_2)$ is a subsolution of \eqref{pde:system:each}.
\vspace{0.1 in}

\noindent {\bf Step 3:} We claim that $(z_1,z_2)$ is a solution of \eqref{pde:system:each}. Moreover, it is a maximal solution of \eqref{pde:system:each}. Indeed, by taking $(\widetilde{u}_1, \widetilde{u}_2)=(z_1, z_2)$ and applying { Step 1}, we have that  there exists $(z_1^*,z_2^*)$ that is a subsolution of \eqref{pde:system:each} with $z_i\leq z_i^* \leq \overline{u}_i,\ i=1,2$, and $z_1^*$ and $z_2^*$ solutions of \eqref{fv'} and \eqref{sv'}, respectively, for the pair $(z_1,z_2)$.
From the definition of the set $\mathcal{A}$, we have that $(z_1^*,z_2^*) \in \mathcal{A}.$ 
Using the fact that $(z_1, z_2)$ is a maximal element of $\mathcal{A}$, we get that $z_i^* \leq z_i$ for $i=1,2$. Hence, $(z_1,z_2)=(z_1^*, z_2^*)$. Therefore, for every $\psi \in H^1(\Omega)$,
$$\int_{\Omega}(\nabla z_1\nabla\psi + c_1z_1\psi)= \int_{\Omega}(\nabla z_1^*\nabla\psi + c_1z_1^*\psi) =\int_{\Omega}f_1(x,z_1,z_2)\psi +\int_{\partial\Omega}g_1(x,z_1,z_2)\psi,$$
and
$$\int_{\Omega}(\nabla z_2\nabla\psi + c_2z_2\psi)=\int_{\Omega}(\nabla z_2^*\nabla\psi + c_2z_2^*\psi) =\int_{\Omega}f_2(x,z_1,z_2)\psi +\int_{\partial\Omega}g_2(x,z_1,z_2)\psi.$$
Thus, $(z_1,z_2)$ is a solution of \eqref{pde:system:each} and $(\underline{u}_1,\underline{u}_2) \le (z_1,z_2) \le (\overline{u}_1,\overline{u}_2)$.

\begin{comment}
    Indeed, observe that since $(z_1, z_2) \in J$ therefore there exists $(z_1^*, z_2^*) \in \mathcal{A}$ such that $z_i\le z_i^*\le \overline{u}_i$, $i=1,2$, and $z_1^*$ and $z_2^*$ solutions of \eqref{fv'} and \eqref{sv'}, respectively, for the pair $(z_1,z_2)$. Thus, $(z_1^*,z_2^*) \in \mathcal{A}$. Using the fact that $(z_1,z_2)$ is a maximal element of $\mathcal{A}$, it follows that $(z_1^*,z_2^*)\le (z_1,z_2)$. Hence, $(z_1^*,z_2^*)= (z_1,z_2)$. Therefore, for every $\psi \in H^1(\Omega)$,
$$\int_{\Omega}(\nabla z_1\nabla\psi + c_1z_1\psi)= \int_{\Omega}(\nabla z_1^*\nabla\psi + c_1z_1^*\psi) =\int_{\Omega}f_1(x,z_1,z_2)\psi +\int_{\partial\Omega}g_1(x,z_1,z_2)\psi,$$
and
$$\int_{\Omega}(\nabla z_2\nabla\psi + c_2z_2\psi)=\int_{\Omega}(\nabla z_2^*\nabla\psi + c_2z_2^*\psi) =\int_{\Omega}f_2(x,z_1,z_2)\psi +\int_{\partial\Omega}g_2(x,z_1,z_2)\psi.$$
Thus, $(z_1,z_2)$ is a solution of \eqref{pde:system:each} and $(\underline{u}_1,\underline{u}_2) \le (z_1,z_2) \le (\overline{u}_1,\overline{u}_2)$. 
\end{comment}

Additionally, let $(u_1,u_2)$ be any solution of \eqref{pde:system:each} such that   $(\underline{u}_1, \underline{u}_2)\leq (u_1,u_2) \leq (\overline{u}_1,\overline{u}_2)$.
Since $(u_1, u_2)$ is a solution of \eqref{pde:system:each}, it is in particular a subsolution of \eqref{pde:system:each}; similarly, $(z_1, z_2)$ is a solution and hence a subsolution of \eqref{pde:system:each}. Therefore, by Proposition \ref{maxsol}, $(v_1,v_2)$ with $v_1=\max\{u_1,z_1\}$ and $v_2=\max\{u_2,z_2\}$ is a subsolution of \eqref{pde:system:each} and by {Step 1}, taking $(\widetilde{u}_1, \widetilde{u}_2) = (v_1, v_2)$, there exists $(\widehat{z}_1, \widehat{z}_2) \in J $ such that
\[
\underline{u}_1 \leq v_1 \leq \widehat{z}_1 \leq \overline{u}_1  \quad \mbox{ and } \quad \widehat{z}_1 \text{ is a solution of \eqref{fv'}} ,
\]
\[
\underline{u}_2 \leq v_2 \leq \widehat{z}_2 \leq \overline{u}_2 \quad \mbox{ and } \quad \widehat{z}_2 \text{ is a solution of \eqref{sv'}}.
\]
Hence $(\widehat{z}_1, \widehat{z}_2) \in \mathcal{A}$.
Since $(z_1, z_2)$ is a maximal element of $\mathcal{A}$,
\begin{equation}
\label{Eq1}    
(\widehat{z}_1, \widehat{z}_2) \leq (z_1, z_2).
\end{equation}
Therefore, $(v_1, v_2) \leq (\widehat{z}_1, \widehat{z}_2) \leq (z_1, z_2)$.

On the other hand $v_1 = \max\{u_1, z_1\} \geq z_1$ and $v_2 = \max\{u_2, z_2\} \geq z_2$, we have
\begin{equation}
    \label{Eq2}
    (v_1, v_2) \geq (z_1, z_2).
\end{equation}
Combining \eqref{Eq1} and \eqref{Eq2},
\[
z_1 = v_1 = \max\{u_1, z_1\} \quad \text{and} \quad z_2 = v_2 = \max\{u_2, z_2\},
\]
which gives $(u_1, u_2) \leq (z_1, z_2)$.
Hence, $(z_1,z_2)$ is a maximal solution of \eqref{pde:system:each}. By a similar approach using Proposition \ref{minsol}, we can show the existence of a minimal solution of \eqref{pde:system:each}.
\qed

\section{Example}
\label{sec:examples}
In this section, we illustrate an example involving nonlinear reaction terms that satisfy conditions {\bf (Q)},  {\bf (H1)}, and {\bf (H2)}.  Consider the eigenvalue problems, $i=1,2$
\begin{equation}
	\label{Stek-Compact21}
	\left\{
	\begin{array}{rcll}
		-\Delta \varphi_i +c_i(x)\varphi_i&=&\mu_i  \varphi_i,   \quad &\mbox{in}\quad \Omega\,,\\
		\dfrac{\partial \varphi_i}{\partial \eta} &=&  \mu_i  \varphi_i\quad &\mbox{on}\quad \partial \Omega
	\end{array}
	\right. 
\end{equation}
where $c_i\in L^{\infty}(\Omega)$  and $c_i>0$ a.e. The existence of the first eigenvalue $\mu_i>0$ and the corresponding positive eigenfunction $0<\varphi_i \in H^1(\Omega)$, $i=1,2$ is well known (see \cite[Theorem 2.2]{amann1976} and \cite[Proposition 2.3]{mavinga2012}).

\begin{thm}\label{example}
Consider the system
\begin{equation}\label{special example} 
\left\{ \begin{array}{rcll}
-\Delta  u_1 +c_1(x)u_1 &=&\lambda f_1(u_2) & \quad \mbox{{\rm in} }\quad  \Om\,;  \\
-\Delta  u_2 +c_2(x)u_2 &=&\lambda f_2(u_1) & \quad \mbox{{\rm in} }\quad  \Om\,;  \\
\dfrac{\p   u_1}{\p \eta}&= &  \lambda g_1(u_2) &  \quad \mbox{{\rm on} } \quad \p\Om\,;\\
\dfrac{\p   u_2}{\p \eta}&= &  \lambda g_2(u_1) &  \quad \mbox{{\rm on} } \quad \p\Om\,,
\end{array}\right.
\end{equation}
where $\lambda>0$ is a parameter and the functions $c_i\in L^{\infty}(\Omega)$ and $c_i>0$ a.e., where the nonlinearities $f_i, g_i:[0, \infty) \to [0, \infty)$ are nondecreasing differentiable functions satisfying $f_i(0)=0 = g_i(0)$, $f_i'(0)>0$, $g_i'(0)>0$, and $\lim\limits_{s \to \infty}\frac{f_i(s)}{s}=0=\lim\limits_{s \to \infty}\frac{g_i(s)}{s}$, $i=1,2$. In addition, suppose that there exist $C_1, C_2>0$ such that $C_2\varphi_2\leq \varphi_1\leq C_1\varphi_2$ in $\overline{\Omega}$.
% \textcolor{blue}{To see this, 
% observe that since $\varphi_i>0$ and continuous on $\overline{\Omega}$, $\min\varphi_i$ exists in $\overline{\Omega},$ fore each $i=1,2.$ Hence, for each $x\in\overline{\Omega}$
 % $$\underbrace{\frac{\min\varphi_1}{\|\varphi_2\|}}_{C_2}\varphi_2(x)\le \varphi_1(x)\le \underbrace{\frac{\|\varphi_1\|}{\min\varphi_2}}_{C_1}\varphi_2(x)$$ }
%  equivalently
%  $$C_2\varphi_2 \le \varphi_1 \le C_1\varphi_2$$}
For $\lambda > \max\left\{\frac{\mu_1C_1}{ab} , \frac{\mu_2}{abC_2}\right\}$, where $a= \min\left\{\sqrt{f_1'(0)},\sqrt{g_1'(0)}\right\}$ and $b= \min\left\{\sqrt{f_2'(0)},\sqrt{g_2'(0)}\right\}$, system \eqref{special example} admits a positive subsolution $(\underline{u}_1,\underline{u}_2)$ and supersolution $(\overline{u}_1, \overline{u}_2)$ satisfying
$$(\underline{u}_1,\underline{u}_2) \le (\overline{u}_1, \overline{u}_2),$$
where $\mu_i, \varphi_i$ are the first eigenvalue and eigenfunction of \eqref{Stek-Compact21}, $i=1,2$, respectively.
\end{thm}

\begin{proof} The proof proceeds through construction of an ordered pair of subsolution and supersolution.

\noindent {\bf Step 1.} In this step, we construct a subsolution. Without loss of generality, we consider $i=1$.
% Let $a=\sqrt{f_1'(0)}$, $b=\sqrt{f_2'(0)}$, $c=\sqrt{g_1'(0)}$ and $d=\sqrt{g_2'(0)}$.
For $\epsilon$ sufficiently small, let $(\underline{u}_1,\underline{u}_2)=(a\varepsilon\varphi_1, b\varepsilon\varphi_2)$, 
%\old{where 
%$$a= \min\left\{\sqrt{f_1'(0)},\sqrt{g_1'(0)}\right\} \quad \text{and} \quad b= \min\left\{\sqrt{f_2'(0)},\sqrt{g_2'(0)}\right\}.$$}

Define 
\begin{align*}
\xi_1(s) = \mu_1 C_1 \frac{a}{b} s- \lambda f_1(s),  \qquad \widetilde{\xi}_1(s) = \mu_1 C_1 \frac{a}{b} s - \lambda g_1(s),\\
\xi_2(s) = \frac{\mu_2}{C_2} \frac{b}{a} s- \lambda f_2(s), \qquad \widetilde{\xi}_2(s) = \frac{\mu_2}{C_2} \frac{b}{a} s - \lambda g_2(s),
\end{align*}

 It follows that $\xi_1(0) = 0=\widetilde{\xi}_1(0)$. Moreover, since $\lambda > \max\left\{\frac{\mu_1C_1}{ab} , \frac{\mu_2}{abC_2}\right\}$, we obtain
\begin{align*}
     \xi_1^{'}(0) = \mu_1C_1 \frac{a}{b} - \lambda f_1^{'}(0) <\lambda ab\frac{a}{b}- \lambda f_1'(0)\leq \lambda f_1'(0)-\lambda f_1'(0) = 0,
\end{align*}
and 
\begin{align*}
     \widetilde{\xi'}_1(0) = \mu_1C_1 \frac{a}{b} - \lambda g_1^{'}(0) <\lambda ab\frac{a}{b}  -\lambda g_1'(0) \leq\lambda g_1'(0)-\lambda g_1'(0)= 0.
\end{align*}

Consequently, we can find $0<\delta \ll 1$ such that $\xi_1(s)<0$ and $\widetilde{\xi}_1(s)<0$ for $s\in (0,\delta)$. Choose $0<\varepsilon\ll 1$ sufficiently small so that $a\varepsilon\varphi_1, ~b\varepsilon\varphi_2 \in (0,\delta)$.

Using \eqref{Stek-Compact21} and the above inequalities, we have for all $0\leq \psi\in H^1(\Omega)$:
\begin{align*}
   \int_{\Omega} \nabla (a\varepsilon\varphi_1)\nabla \psi +\int_{\Omega} (a\varepsilon\varphi_1)c_1\psi&=\mu_1\int_{\Omega}(a\varepsilon\varphi_1)\,\psi+ \mu_1\int_{\partial\Omega} (a\varepsilon\varphi_1)\,\psi\,\\
   &\leq \mu_1C_1\int_{\Omega}(a\varepsilon\varphi_2)\,\psi+ \mu_1C_1\int_{\partial\Omega} (a\varepsilon\varphi_2)\,\psi\,\\
&\leq\lambda\int_{\Omega}f_1(b\varepsilon\varphi_2)\,\psi+ \lambda\int_{\partial\Omega} g_1(b\varepsilon\varphi_2)\psi.
\end{align*}
%{\color{red} According to the definition of $\xi_i$, the numbers $a$ and $b$ should be outside the argument of $f_1$ and $f_2$, respectively. For the crossing in the variables, we need to modify the function $\xi_i$'s to guarantee that $f_1$ depends on $b\epsilon \varphi$ and $f_2$ depends on $a\epsilon \varphi.$} \textcolor{magenta}{since a and b are constants and $f_i$s are defined on $(0,\delta)$, it appears that all that is needed is to ensure that $a\varepsilon\varphi \in (0,\delta)$.}
A similar argument establishes that 
\begin{align*}
    \int_{\Omega} \nabla (b\varepsilon\varphi_2)\nabla \psi +\int_{\Omega} (b\varepsilon\varphi_2)c_2\,\psi
   &\le \lambda\int_{\Omega}f_2(a\varepsilon\varphi_1)\,\psi+ \lambda\int_{\partial\Omega} g_2(a\varepsilon\varphi_1)\,\psi.
\end{align*}
  % \textcolor{red}{ Removed from above
  % \begin{align*}
  %     &=\mu_2\int_{\Omega}(b\varepsilon\varphi_2)\,\psi+ \mu_2\int_{\partial\Omega} (b\varepsilon\varphi_2)\,\psi\,\\
  %  &\leq \frac{\mu_2}{C_2}\int_{\Omega}(b\varepsilon\varphi_1)\,\psi+ \frac{\mu_2}{C_2}\int_{\partial\Omega} (b\varepsilon\varphi_1)\,\psi\,
  % \end{align*}
  % }
Therefore, $(\underline{u}_1,\underline{u}_2) = (a\varepsilon\varphi_1, b\varepsilon\varphi_2)$ is a weak subsolution of \eqref{special example}.

\noindent \textbf{Step 2.} 
For $\widetilde M\gg 1$, we demonstrate that $(\overline{u}_1, \overline{u}_2) = (\widetilde M\phi_1^*, \widetilde M\phi_2^*)$ is a supersolution of \eqref{special example}, where $\phi_i^*:=\frac{\phi_i}{\|\phi_i\|_{\infty}} \in C^{2,\alpha}(\Omega)\cap C^{1,\alpha}(\overline\Omega)$, $i=1,2$, and $\phi_i$ is the unique positive solution of 
 \begin{equation}
 	\label{Stek22}
 	\left\{
 	\begin{array}{rcll}
 		-\Delta \phi_i +c_i(x)\phi_i&=&1,   \quad &\mbox{in}\quad \Omega\,,\\
 		\dfrac{\partial \phi_i}{\partial \eta} &=& 1\quad &\mbox{on}\quad \partial \Omega.
 	\end{array}
 	\right. 
 \end{equation}
%{\color{black}In order to have normalized eigenfunctions, we set $\phi_i^* := \frac{\phi_i}{\|\phi_i\|_\infty}$ for $i = 1, 2$, so that $\|\phi_1^*\|_\infty = \|\phi_2^*\|_\infty = 1$. Then }
% \textcolor{red}{Observe that the normalized eigenfunctions $\phi_i^*$, $i=1,2$ satisfy 
% \begin{equation}
% \label{Stek33}
% \left\{
%  \begin{array}{rcll}	
% -\Delta \phi_i^* + c_i(x)\phi_i^* &= \dfrac{1}{\|\phi_i\|_\infty}, & \text{in } \Omega, \\[8pt]
% \dfrac{\partial \phi_i^*}{\partial \eta} &= \dfrac{1}{\|\phi_i\|_\infty}, & \text{on } \partial\Omega.
%  \end{array}
%  \right. 
% \end{equation}
% }

\begin{comment}
 {\color{red} In order to have normalized eigenfunctions, we can always consider the problem
\begin{equation}
	
	\left\{
	\begin{array}{rcll}
		-\Delta \phi_i^* +c_i(x)\phi_i^*&=&\frac{1}{\|\phi_i\|_{\infty}},   \quad &\mbox{in}\quad \Omega\,,\\
		\frac{\partial \phi_i^*}{\partial \eta} &=& \frac{1}{\|\phi_i\|_{\infty}}\quad &\mbox{on}\quad \partial \Omega.
	\end{array}
	\right. 
\end{equation}
with $\|\phi_1^*\|_{\infty}=\|\phi_2^*\|_{\infty}=1.$}
\end{comment}

% \textcolor{red}{Observe that $\overline{f}_i(t)=\underset{ s\in [0,t]}{\max}f_i(s)$ and $\overline{g}_i(t)=\underset{ s\in [0,t]}{\max}g_i(s)$ are nondecreasing functions satisfying $f_i(s)\le \overline{f}_i(t)$ and $g_i(s)\le \overline{g}_i(t)$ for $t>0$.} 
Denote $\overline{f}_i(t)=\underset{ s\in [0,t]}{\max}f_i(s)$ and $\overline{g}_i(t)=\underset{ s\in [0,t]}{\max}g_i(s)$. Since $\lim\limits_{s \to \infty}\frac{f_i(s)}{s}=0=\lim\limits_{s \to \infty}\frac{g_i(s)}{s}$, it follows that 
$$\underset{t\to\infty}{\lim}\frac{\overline{f}_i(t)}{t}=0=\underset{t\to\infty}{\lim}\frac{\overline{g}_i(t)}{t}.$$

Hence, for $i=1,2$ and $\widetilde M\gg 1$ sufficiently large, we have
\begin{equation}
\label{sup2}    \frac{\overline{f}_i(\widetilde M\|\phi_i^*\|_{\infty})}{\widetilde M\|\phi_i^*\|_{\infty}} < \frac{1}{\lambda\|\phi_i\|_{\infty}} \quad \text{and} \quad \frac{\overline{g}_i(\widetilde M\|\phi_i^*\|_{\infty})}{\widetilde M\|\phi_i^*\|_{\infty}} < \frac{1}{\lambda\|\phi_i\|_{\infty}}.
\end{equation}
Using  \eqref{Stek22}, \eqref{sup2}, and the fact that  $\phi_i^*:=\frac{\phi_i}{\|\phi_i\|_{\infty}}$, we obtain
\begin{align*}
   \int_{\Omega} \nabla \widetilde M\phi_1^*\nabla \psi +\int_{\Omega} c_1\widetilde M\phi_1^*\psi&=\int_{\Omega} \frac{\widetilde M}{\|\phi_1\|_{\infty}} \,\psi+ \int_{\partial\Omega} \frac{\widetilde M}{\|\phi_1\|_{\infty}}\,\psi\,\\
    &\ge \lambda\int_{\Omega} \overline{f}_1(\widetilde M\|\phi_1^*\|_{\infty})\,\psi+ \lambda\int_{\partial\Omega} \overline{g}_1(\widetilde M\|\phi_1^*\|_{\infty})\,\psi\,\\
    &= \lambda\int_{\Omega} \overline{f}_1(\widetilde M\|\phi_2^*\|_{\infty})\,\psi+ \lambda\int_{\partial\Omega} \overline{g}_1(\widetilde M\|\phi_2^*\|_{\infty})\,\psi\,\\
    &\ge \lambda\int_{\Omega}f_1(\widetilde M\phi_2^*)\,\psi+ \lambda\int_{\partial\Omega} g_1(\widetilde M\phi_2^*)\,\psi.
\end{align*}
Similarly, 
\begin{align*}
   \int_{\Omega} \nabla \widetilde M\phi_2^*\nabla \psi +\int_{\Omega} c_2\widetilde M\phi_2^*\psi \geq\lambda\int_{\Omega}f_2(\widetilde M\phi_1^*)\,\psi+ \lambda\int_{\partial\Omega} g_2(\widetilde M\phi_1^*)\,\psi.
\end{align*}
Thus, $(\overline{u}_1, \overline{u}_2) = (\widetilde M\phi_1^*, \widetilde M\phi_2^*)$ is a weak supersolution of \eqref{special example}. 

\noindent 
Finally, to show that the sub and supersolution is ordered, choose $\widetilde M > \max \left\{ \dfrac{a\epsilon \varphi_1}{\displaystyle\min_{x \in \Omega} \phi_1^*}, \dfrac{b\epsilon \varphi_2}{\displaystyle\min_{x \in \Omega} \phi_2^*}\right\}$ such that $a\varepsilon\varphi_1 \le  \widetilde M\phi_1^*$ and $b\varepsilon\varphi_2 \le  \widetilde M\phi_2^*$, which yields 
$$(\underline{u}_1,\underline{u}_2) = (a\varepsilon\varphi_1, b\varepsilon\varphi_2) \le (\widetilde M\phi_1^*, \widetilde M\phi_2^*)=(\overline{u}_1, \overline{u}_2).$$
Therefore, by Theorem 1.1, there exists a positive solution $(u_1, u_2)$ to system \eqref{special example} satisfying $(0,0)< (\underline{u}_1,\underline{u}_2) \leq (u_1, u_2) \leq (\overline{u}_1, \overline{u}_2)$. 
\end{proof}

\section*{Appendix}
\label{sec:appdx}
{In this section, we prove that the maximum of two subsolutions of \eqref{pde:system:each} remains a subsolution of \eqref{pde:system:each} and the minimum of two supersolutions of \eqref{pde:system:each} remains a supersolution of \eqref{pde:system:each}. This extends Kato's inequality for single equations proved in Section \ref{sec:prelim} (see Proposition \ref{prop:main_kato} and Corollary \ref{cor}) to quasimonotone systems.}
%{\color{blue} Recall the set $J$ \eqref{def:J} and \eqref{pde:system:decoupled} respectively and the operator $T$ defined in the proof of Theorem \ref{th:monotone}.}

\begin{lem}
\label{lem:5.1}
Let $(\alpha_1, \alpha_2)$ and $(\beta_1, \beta_2) \in \mathcal{A}$ be any two subsolutions of \eqref{pde:system:each}. Then the pair $(\max\{\alpha_1, \beta_1\}, \max\{\alpha_2, \beta_2\})$ is a subsolution of \eqref{pde:system:each}.
\end{lem}
\begin{proof}
Since $(\alpha_1, \alpha_2)$ and $(\beta_1, \beta_2)$ belong to $\mathcal{A}$, there exist corresponding functions such that each pair satisfies the subsolution conditions. Specifically, 
For $(\alpha_1, \alpha_2)$:
\begin{align*}
\int_\Omega \nabla \alpha_1 \nabla \psi + \int_\Omega c_1(x)\alpha_1 \psi &\leq \int_\Omega f_1(x, \alpha_1, \alpha_2) \psi + \int_{\partial \Omega} g_1(x, \alpha_1, \alpha_2) \psi \\
\int_\Omega \nabla \alpha_2 \nabla \psi + \int_\Omega c_2(x)\alpha_2 \psi &\leq \int_\Omega f_2(x, \alpha_1, \alpha_2) \psi + \int_{\partial \Omega} g_2(x, \alpha_1, \alpha_2) \psi
\end{align*}
For $(\beta_1, \beta_2)$:
\begin{align*}
\int_\Omega \nabla \beta_1 \nabla \psi + \int_\Omega c_1(x)\beta_1 \psi &\leq \int_\Omega f_1(x, \beta_1, \beta_2) \psi + \int_{\partial \Omega} g_1(x, \beta_1, \beta_2) \psi \\
\int_\Omega \nabla \beta_2 \nabla \psi + \int_\Omega c_2(x)\beta_2 \psi &\leq \int_\Omega f_2(x, \beta_1, \beta_2) \psi + \int_{\partial \Omega} g_2(x, \beta_1, \beta_2) \psi
\end{align*}
Now define $\gamma_1 := \max\{\alpha_1, \beta_1\}$ and $\gamma_2 := \max\{\alpha_2, \beta_2\}$.
By the quasimonotonicity condition {\bf (Q)}, we have:
\begin{enumerate}
\item $f_1(x, \alpha_1, \alpha_2) \leq f_1(x, \alpha_1, \gamma_2)$ and $f_1(x, \beta_1, \beta_2) \leq f_1(x, \beta_1, \gamma_2)$
\item $f_2(x, \alpha_1, \alpha_2) \leq f_2(x, \gamma_1, \alpha_2)$ and $f_2(x, \beta_1, \beta_2) \leq f_2(x, \gamma_1, \beta_2)$
\end{enumerate}
Similarly for the boundary terms:
\begin{enumerate}
\item $g_1(x, \alpha_1, \alpha_2) \leq g_1(x, \alpha_1, \gamma_2)$ and $g_1(x, \beta_1, \beta_2) \leq g_1(x, \beta_1, \gamma_2)$
\item $g_2(x, \alpha_1, \alpha_2) \leq g_2(x, \gamma_1, \alpha_2)$ and $g_2(x, \beta_1, \beta_2) \leq g_2(x, \gamma_1, \beta_2)$
\end{enumerate}
Let
\begin{equation*}
h_1(x) := \begin{cases}
f_1(x, \alpha_1, \gamma_2) & \text{if } \alpha_1(x) > \beta_1(x) \text{ a.e. } x \in \Omega \\
f_1(x, \beta_1, \gamma_2) & \text{if } \alpha_1(x) \leq \beta_1(x) \text{ a.e. } x \in \Omega
\end{cases}
\end{equation*}
\begin{equation*}
\widetilde{h}_1(x) := \begin{cases}
g_1(x, \alpha_1, \gamma_2) & \text{if } \alpha_1(x) > \beta_1(x) \text{ a.e. } x \in \partial \Omega \\
g_1(x, \beta_1, \gamma_2) & \text{if } \alpha_1(x) \leq \beta_1(x) \text{ a.e. } x \in \partial \Omega
\end{cases}
\end{equation*}
and 
\begin{equation*}
h_2(x) := \begin{cases}
f_2(x, \gamma_1, \alpha_2) & \text{if } \alpha_2(x) > \beta_2(x) \text{ a.e. } x \in \Omega \\
f_2(x, \gamma_1, \beta_2) & \text{if } \alpha_2(x) \leq \beta_2(x) \text{ a.e. } x \in \Omega
\end{cases}
\end{equation*}
\begin{equation*}
\widetilde{h}_2(x) := \begin{cases}
g_2(x, \gamma_1, \alpha_2) & \text{if } \alpha_2(x) > \beta_2(x) \text{ a.e. } x \in \partial \Omega \\
g_2(x, \gamma_1, \beta_2) & \text{if } \alpha_2(x) \leq \beta_2(x) \text{ a.e. } x \in \partial \Omega
\end{cases}
\end{equation*}
Notice that {for $i=1,2, \; \alpha_i$ and $\beta_i$ satisfy 
\begin{equation*}
\int_\Omega \nabla v \nabla \psi + \int_\Omega c_i(x)v \psi \leq \int_\Omega h_i(x) \psi + \int_{\partial \Omega} \widetilde{h}_i(x) \psi
\end{equation*}
}
{Since $(\alpha_i, \beta_i)$ satisfy $\underline{u}_i \le\alpha_i, \beta_i \le \overline{u}_i$ and $\gamma_2 = \max\{\alpha_2, \beta_2\} \in [\underline{u}_2, \overline{u}_2]$, condition {\bf (H1)} gives $|h_i(x)| \le K_i(x) \in L^{\frac{2N}{N+2}}(\Omega)$ and condition {\bf (H2)} gives $|\widetilde{h}_i(x)| \le \widetilde{K}_i(x) \in L^{\frac{2(N-1)}{N}}(\partial\Omega)$, so the hypotheses of Proposition~\ref{prop:main_kato} are satisfied.}
{Therefore, by} {Proposition \ref{prop:main_kato}}, we have that $(\gamma_1, \gamma_2)$ satisfies:
\begin{align*}
\int_\Omega \nabla \gamma_1 \nabla \psi + \int_\Omega c_1(x)\gamma_1 \psi &\leq \int_\Omega h_1(x) \psi + \int_{\partial \Omega} \widetilde{h}_1(x) \psi \\
\int_\Omega \nabla \gamma_2 \nabla \psi + \int_\Omega c_2(x)\gamma_2 \psi &\leq \int_\Omega h_2(x) \psi + \int_{\partial \Omega} \widetilde{h}_2(x) \psi
\end{align*}
By the definitions of $h_1, h_2, \widetilde{h}_1, \widetilde{h}_2$ and the quasimonotonicity condition, we have:
\begin{enumerate}
\item $h_1(x) \leq f_1(x, \gamma_1, \gamma_2)$ and $\widetilde{h}_1(x) \leq g_1(x, \gamma_1, \gamma_2)$
\item $h_2(x) \leq f_2(x, \gamma_1, \gamma_2)$ and $\widetilde{h}_2(x) \leq g_2(x, \gamma_1, \gamma_2)$
\end{enumerate}
Therefore:
\begin{align*}
\int_\Omega \nabla \gamma_1 \nabla \psi + \int_\Omega c_1(x)\gamma_1 \psi &\leq \int_\Omega f_1(x, \gamma_1, \gamma_2) \psi + \int_{\partial \Omega} g_1(x, \gamma_1, \gamma_2) \psi \\
\int_\Omega \nabla \gamma_2 \nabla \psi + \int_\Omega c_2(x)\gamma_2 \psi &\leq \int_\Omega f_2(x, \gamma_1, \gamma_2) \psi + \int_{\partial \Omega} g_2(x, \gamma_1, \gamma_2) \psi
\end{align*}
This shows that $(\gamma_1, \gamma_2) = (\max\{\alpha_1, \beta_1\}, \max\{\alpha_2, \beta_2\})$ is a subsolution of \eqref{pde:system:each}.
\end{proof}
\begin{prop}\label{maxsol}
Suppose that $(u_1, u_2)$ and $(v_1, v_2)$ are two solutions of \eqref{pde:system:each} such that $(\underline{u}_1, \underline{u}_2) \leq (u_1, u_2), (v_1, v_2) \leq (\overline{u}_1, \overline{u}_2)$, where $(\underline{u}_1, \underline{u}_2)$ and $(\overline{u}_1, \overline{u}_2)$ are sub- and supersolutions of \eqref{pde:system:each}, respectively. Then $(\max\{u_1, v_1\}, \max\{u_2, v_2\})$ is a subsolution of \eqref{pde:system:each}.
\end{prop}
\begin{proof}
Since $(u_1, u_2)$ and $(v_1, v_2)$ are solutions of \eqref{pde:system:each}, they satisfy
\begin{align*}
\int_\Omega \nabla u_i \nabla \psi + \int_\Omega c_i(x)u_i \psi &= \int_\Omega f_i(x, u_1, u_2) \psi + \int_{\partial \Omega} g_i(x, u_1, u_2) \psi, \\
\int_\Omega \nabla v_i \nabla \psi + \int_\Omega c_i(x)v_i \psi &= \int_\Omega f_i(x, v_1, v_2) \psi + \int_{\partial \Omega} g_i(x, v_1, v_2) \psi,
\end{align*}
for all $\psi \in H^1(\Omega)$ and $i = 1, 2$.
Let $\gamma_1 := \max\{u_1, v_1\}$ and $\gamma_2 := \max\{u_2, v_2\}$. 
By the quasimonotonicity condition {\bf(Q)}, we have
\begin{align*}
f_1(x, u_1, u_2) &\leq f_1(x, u_1, \gamma_2), \quad & 
g_1(x, u_1, u_2) &\leq g_1(x, u_1, \gamma_2), \\
f_1(x, v_1, v_2) &\leq f_1(x, v_1, \gamma_2), \quad & 
g_1(x, v_1, v_2) &\leq g_1(x, v_1, \gamma_2), \\
f_2(x, u_1, u_2) &\leq f_2(x, \gamma_1, u_2), \quad & 
g_2(x, u_1, u_2) &\leq g_2(x, \gamma_1, u_2), \\
f_2(x, v_1, v_2) &\leq f_2(x, \gamma_1, v_2), \quad & 
g_2(x, v_1, v_2) &\leq g_2(x, \gamma_1, v_2).
\end{align*}
Define the composite functions as follows
\begin{equation*}
\widetilde{f}_1(x) := \begin{cases}
f_1(x, u_1, \gamma_2) & \text{if } u_1(x) > v_1(x) \text{ a.e. } x \in \Omega \\
f_1(x, v_1, \gamma_2) & \text{if } u_1(x) \leq v_1(x) \text{ a.e. } x \in \Omega
\end{cases}
\end{equation*}
\begin{equation*}
\widetilde{g}_1(x) := \begin{cases}
g_1(x, u_1, \gamma_2) & \text{if } u_1(x) > v_1(x) \text{ a.e. } x \in \partial \Omega \\
g_1(x, v_1, \gamma_2) & \text{if } u_1(x) \leq v_1(x) \text{ a.e. } x \in \partial \Omega
\end{cases}
\end{equation*}
\begin{equation*}
\widetilde{f}_2(x) := \begin{cases}
f_2(x, \gamma_1, u_2) & \text{if } u_2(x) > v_2(x) \text{ a.e. } x \in \Omega \\
f_2(x, \gamma_1, v_2) & \text{if } u_2(x) \leq v_2(x) \text{ a.e. } x \in \Omega
\end{cases}
\end{equation*}
\begin{equation*}
\widetilde{g}_2(x) := \begin{cases}
g_2(x, \gamma_1, u_2) & \text{if } u_2(x) > v_2(x) \text{ a.e. } x \in \partial \Omega \\
g_2(x, \gamma_1, v_2) & \text{if } u_2(x) \leq v_2(x) \text{ a.e. } x \in \partial \Omega
\end{cases}
\end{equation*}
Since $u_1$ and $v_1$ are solutions, they satisfy the subsolution inequalities
\begin{equation*}
\int_\Omega \nabla w \nabla \psi + \int_\Omega c_1(x)w \psi \leq \int_\Omega \widetilde{f}_1(x) \psi + \int_{\partial \Omega} \widetilde{g}_1(x) \psi
\end{equation*}
for $w = u_1, v_1$.
{Since $u_i, v_i \in [\underline{u}_i, \overline{u}_i]$ and $\gamma_2 \in 
[\underline{u}_2, \overline{u}_2]$, condition {\bf (H1)} gives 
$|\widetilde{f}_i(x)| \le K_i(x) \in L^{\frac{2N}{N+2}}(\Omega)$ and condition \bf (H2)} 
gives $|\widetilde{g}_i(x)| \le \widetilde{K}_i(x) \in L^{\frac{2(N-1)}{N}}(\partial\Omega)$, 
so the hypotheses of Proposition~\ref{prop:main_kato} are satisfied. Therefore, by Proposition \ref{prop:main_kato} (Extended Kato's inequality), $\gamma_1 = \max\{u_1, v_1\}$ satisfies
\begin{equation*}
\int_\Omega \nabla \gamma_1 \nabla \psi + \int_\Omega c_1(x)\gamma_1 \psi \leq \int_\Omega \widetilde{f}_1(x) \psi + \int_{\partial \Omega} \widetilde{g}_1(x) \psi
\end{equation*}
By the quasimonotonicity conditions and the definitions of $\widetilde{f}_1$ and $\widetilde{g}_1$, we have that
\begin{equation*}
\widetilde{f}_1(x) \leq f_1(x, \gamma_1, \gamma_2) \quad \text{and} \quad \widetilde{g}_1(x) \leq g_1(x, \gamma_1, \gamma_2)
\end{equation*}
Similarly for $\gamma_2$. Therefore,
\begin{align*}
\int_\Omega \nabla \gamma_1 \nabla \psi + \int_\Omega c_1(x)\gamma_1 \psi &\leq \int_\Omega f_1(x, \gamma_1, \gamma_2) \psi + \int_{\partial \Omega} g_1(x, \gamma_1, \gamma_2) \psi \\
\int_\Omega \nabla \gamma_2 \nabla \psi + \int_\Omega c_2(x)\gamma_2 \psi &\leq \int_\Omega f_2(x, \gamma_1, \gamma_2) \psi + \int_{\partial \Omega} g_2(x, \gamma_1, \gamma_2) \psi
\end{align*}
This shows that $(\gamma_1, \gamma_2) = (\max\{u_1, v_1\}, \max\{u_2, v_2\})$ is a subsolution of \eqref{pde:system:each}.
\end{proof}

Next, we show that if $(u_1, u_2)$ and $(v_1, v_2)$ are two supersolutions of \eqref{pde:system:each}, then the pair of functions $(\min\{u_1, v_1\}, \min\{u_2, v_2\})$ is a supersolution of \eqref{pde:system:each}. For that purpose, consider the equations
\begin{equation}\label{eq:super_I}
\left\{
	\begin{array}{rcll}
    -\Delta u_1 + c_1(x)u_1 &=& f_1(x, u_1, \widehat{u}_2) & \text{in} \; \Omega,\\
    \dfrac{\partial u_1}{\partial \eta} &=& g_1(x, u_1, \widehat{u}_2) & \text{on}\; \partial\Omega,
    \end{array}
	\right.
\end{equation}
and
\begin{equation}\label{eq:super_II}
 \left\{
	\begin{array}{rcll}
    -\Delta u_2 + c_2(x)u_2 &=& f_2(x, \widehat{u}_1, u_2) & \text{in} \; \Omega,\\
    \dfrac{\partial u_2}{\partial \eta} &=& g_2(x, \widehat{u}_1, u_2) & \text{on}\; \partial\Omega.
    \end{array}
	\right.
\end{equation}
Let $\mathcal{B}$ be the set consisting of $(w_1,w_2) \in (H^1(\Omega))^2$ such that there exists a supersolution $(\widehat{u}_1, \widehat{u}_2)$ of \eqref{pde:system:each} satisfying
\begin{equation*}
%\label{eq:set_B}
(\underline{u}_1, \underline{u}_2) \le (w_1,w_2) \le (\widehat{u}_1, \widehat{u}_2) \le (\overline{u}_1, \overline{u}_2),
\end{equation*}
where $w_1$ and $w_2$ are solutions of \eqref{eq:super_I} and \eqref{eq:super_II}, respectively, for the pair $(\widehat{u}_1, \widehat{u}_2)$.
\begin{lem}
Let $(\alpha_1, \alpha_2)$ and $(\beta_1, \beta_2) \in \mathcal{B}$ be any two supersolutions of \eqref{pde:system:each}. Then the pair $(\min\{\alpha_1, \beta_1\}, \min\{\alpha_2, \beta_2\})$ is a supersolution of \eqref{pde:system:each}.
\end{lem}
\begin{proof}
The proof follows the same structure as Lemma~\ref{lem:5.1}, with the following modifications: (i) supersolution inequalities ($\geq$) replace subsolution inequalities ($\leq$), (ii) $\min$ operations replace $\max$ operations, (iii) the case definitions in the composite functions use conditions of the form ``$\alpha_1(x) < \beta_1(x)$'' instead of ``$\alpha_1(x) > \beta_1(x)$'', and (iv) we apply the minimum property from Corollary~\ref{cor} instead of the maximum property from Proposition~\ref{prop:main_kato}. The proof is otherwise analogous.
\end{proof}
\begin{prop}
\label{minsol}
Suppose that $(u_1, u_2)$ and $(v_1, v_2)$ are two solutions of \eqref{pde:system:each} such that $(\underline{u}_1, \underline{u}_2) \leq (u_1, u_2), (v_1, v_2) \leq (\overline{u}_1, \overline{u}_2)$, where $(\underline{u}_1, \underline{u}_2)$ and $(\overline{u}_1, \overline{u}_2)$ are sub- and supersolutions of \eqref{pde:system:each}, respectively. Then $(\min\{u_1, v_1\}, \min\{u_2, v_2\})$ is a supersolution of \eqref{pde:system:each}.
\end{prop}
\begin{proof}
This combines the modifications from the previous lemma (supersolution inequalities, $\min$ operations, minimum property) with the adjustment from Proposition \ref{maxsol} (starting with solution equalities rather than supersolution inequalities). The proof follows analogously to the previous results.
\end{proof}

\section*{Acknowledgment} 

This collaboration took place as part of the Women in Analysis (WoAN) workshop, hosted by the Banff International Research Station (BIRS) in May 2025, workshop 25w5452. We thank BIRS for its support and hospitality. 
The first author was supported by the AMS-Simons Travel Award and AWM-NSF mentoring travel grant and the third author  was  supported by the AMS-Simons Research Enhancement Grant for PUI Faculty.

We are indebted to the anonymous referee for a careful reading of the manuscript
and  for the valuable comments and suggestions.\\

\noindent {\bf Author Contributions}

S.B., B.D., N.M., and M.O. wrote the main manuscript text. All authors reviewed the manuscript.\\

\noindent {\bf Data availability} 

No datasets were generated or analyzed during the current study.\\

\section*{Declarations}
{\bf Conflict of interest}: The authors declare no conflict of interest.

%\bibliographystyle{amsplain}
%\bibliography{references}

\end{document}